\documentclass[12pt,a4paper]{amsart}

\usepackage{amsmath,amsfonts,amssymb,mathtools,extarrows}
\usepackage{xcolor}

\usepackage{color}
\usepackage{graphicx}

\usepackage{enumerate}  

\usepackage{cancel}

\usepackage[hmargin=2cm,vmargin=2cm]{geometry}

\usepackage[hypertexnames=false,hyperfootnotes=false,colorlinks=true,linkcolor=blue,%
citecolor=purple,filecolor=magenta,urlcolor=cyan,unicode,linktocpage=true,pagebackref=false]{hyperref}
\usepackage{nameref,zref-xr}     

\usepackage{tikz}

\newcommand{\mbZ}{\mathbb Z}
\newcommand{\mbC}{\mathbb C}

\newcommand{\oM}{\overline{\mathcal M}}

\newcommand{\og}{\overline g}
\newcommand{\oh}{\overline h}
\newcommand{\hLambda}{\widehat\Lambda}

\def\oM{{\overline{\mathcal{M}}}}
\def\CP{{{\mathbb C}{\mathbb P}}}
\renewcommand{\Im}{\mathrm{Im}}
\def\mbQ{{\mathbb Q}}
\def\d{{\partial}}

\newcommand{\<}{\left<}
\renewcommand{\>}{\right>}
\newcommand{\eps}{\varepsilon}

\newcommand{\hcA}{\widehat{\mathcal A}}
\newcommand{\DR}{\mathrm{DR}}

\newcommand{\Coef}{\mathrm{Coef}}

\renewcommand{\top}{\mathrm{top}}

\newcommand{\st}{\mathbf{H}}

\newcommand{\mcF}{\mathcal{F}}

\newcommand{\of}{\overline{f}}

\DeclareMathOperator{\res}{{res}}

\newcommand{\oH}{\overline{H}}

\newcommand{\ev}{\mathrm{ev}}

\newcommand{\cL}{\mathcal{L}}

\renewcommand{\st}{\mathrm{st}}

\newcommand{\virt}{\mathrm{virt}}
\newcommand{\od}{\overline{d}}
\newcommand{\cE}{\mathcal{E}}
\newcommand{\mbE}{\mathbb{E}}
\newcommand{\hcB}{\widehat{\mathcal{B}}}
\newcommand{\tphi}{\widetilde{\phi}}
\newcommand{\Ker}{\mathrm{Ker}}

\newtheorem{theorem}{Theorem}[section]

\newtheorem{lemma}[theorem]{Lemma}

\theoremstyle{remark}

\newtheorem{remark}[theorem]{Remark}

\theoremstyle{definition}

\usepackage{color}

\numberwithin{equation}{section}

\begin{document}

\title[Quantum intersection numbers and the Gromov--Witten invariants of $\CP^1$]{Quantum intersection numbers and the Gromov--Witten invariants of $\CP^1$}

\author{Xavier Blot}
\address{X. Blot:\newline
Korteweg--de Vries Institute for Mathematics, University of Amsterdam, Postbus 94248, 1090GE Amsterdam, The Netherlands}
\email{xavierblot1@gmail.com}

\author{Alexandr Buryak}
\address{A. Buryak:\newline 
Faculty of Mathematics, National Research University Higher School of Economics, Usacheva str. 6, Moscow, 119048, Russian Federation;\smallskip\newline 
Igor Krichever Center for Advanced Studies, Skolkovo Institute of Science and Technology, Bolshoy Boulevard 30, bld. 1, Moscow, 121205, Russian Federation}
\email[corresponding author]{aburyak@hse.ru}

\begin{abstract}
The notion of a quantum tau-function for a natural quantization of the KdV hierarchy was introduced in a work of Dubrovin, Gu\'er\'e, Rossi, and the second author. A certain natural choice of a quantum tau-function was then described by the first author, the coefficients of the logarithm of this series are called the quantum intersection numbers. Because of the Kontsevich--Witten theorem, a part of the quantum intersection numbers coincides with the classical intersection numbers of psi-classes on the moduli spaces of stable algebraic curves. In this paper, we relate the quantum intersection numbers to the stationary relative Gromov--Witten invariants of $(\CP^1,0,\infty)$ with an insertion of a Hodge class. Using the Okounkov--Pandharipande approach to such invariants (with the trivial Hodge class) through the infinite wedge formalism, we then give a short proof of an explicit formula for the ``purely quantum'' part of the quantum intersection numbers, found by the first author, which in particular relates these numbers to the one-part double Hurwitz numbers.
\end{abstract}

\date{\today}

\maketitle

\section{Introduction}

The starting point of our considerations is Witten's conjecture~\cite{Wit91}, proved by Kontsevich~\cite{Kon92}, which opened a new direction of research relating the topology of the moduli space~$\oM_{g,n}$ of stable algebraic curves of genus $g$ with $n$ marked points to the theory of integrable systems. Denote by $\psi_i\in H^2(\oM_{g,n},\mbQ)$ the first Chern class of the line bundle $\cL_i$ over~$\oM_{g,n}$ formed by the cotangent lines at the $i$-th marked point on stable curves. The classes~$\psi_i$ are called the {\it psi-classes}. The \emph{intersection numbers} on $\oM_{g,n}$ are defined by
$$
\<\tau_{d_1}\tau_{d_2}\ldots\tau_{d_n}\>_g:=\int_{\oM_{g,n}}\psi_1^{d_1}\psi_2^{d_2}\cdots\psi_n^{d_n}\in\mbQ,\quad d_1,\ldots,d_n\in\mbZ_{\ge 0},
$$ 
with the convention that the integral is zero if $\sum d_i\ne\dim\oM_{g,n}=3g-3+n$. Consider the generating series of intersection numbers
$$
\mcF(t_0,t_1,\ldots,\eps):=\sum_{g,n\ge 0}\eps^{2g}\sum_{d_1,\ldots,d_n\ge 0}\<\tau_{d_1}\tau_{d_2}\ldots\tau_{d_n}\>_g\frac{t_{d_1}t_{d_2}\ldots t_{d_n}}{n!}\in\mbQ[[t_0,t_1,\ldots,\eps]].
$$
Witten's conjecture~\cite{Wit91}, proved by Kontsevich~\cite{Kon92}, states that $u^\top:=\frac{\d^2\mcF}{\d t_0^2}$ is a solution of the Korteweg--de Vries (KdV) hierarchy (we identify $x=t_0$)
\begin{align*}
\frac{\d u}{\d t_1}&=u u_x+\frac{\eps^2}{12}u_{xxx},\\
\frac{\d u}{\d t_2}&=\frac{u^2 u_x}{2}+\eps^2\left(\frac{u u_{xxx}}{12}+\frac{u_xu_{xx}}{6}\right)+\eps^4\frac{u_{xxxxx}}{240},\\
&\hspace{2mm}\vdots
\end{align*}
Note that the required solution of the hierarchy is specified by the initial condition $u^\top|_{t_{\ge 1}=0}=x$. The generating series $\mcF$ can be uniquely reconstructed from $u^\top=\frac{\d^2\mcF}{\d t_0^2}$ using the string equation
$$
\frac{\d\mcF}{\d t_0}=\sum_{k\ge 0}t_{k+1}\frac{\d\mcF}{\d t_k}+\frac{t_0^2}{2}.
$$

\medskip

The KdV hierarchy is Hamiltonian,
$$
\frac{\d u}{\d t_p}=\{u,\oh_p\},
$$
with
\begin{align*}
\oh_1&=\int\left(\frac{u^3}{6}+\frac{\eps^2}{24}u u_{xx}\right)dx,\\
\oh_2&=\int\left(\frac{u^4}{24}+\eps^2\frac{u^2 u_{xx}}{48}+\eps^4\frac{u u_{xxxx}}{480}\right)dx,\\
&\hspace{2mm}\vdots
\end{align*}
and the Poisson bracket $\{\cdot,\cdot\}$ on the space of local functionals given by $\{\oh,\og\}=\int\frac{\delta\oh}{\delta u}\d_x\frac{\delta\og}{\delta u}dx$. Using the Hamiltonians~$\oh_d$, the statement of the Kontsevich--Witten theorem can be equivalently written as
$$
\eps^{2g}\<\tau_0\tau_{d_1}\cdots\tau_{d_n}\>_g=\left.\left\{\left\{\ldots\left\{\left\{\frac{\delta\oh_{d_1}}{\delta u},\oh_{d_2}\right\},\oh_{d_3}\right\},\ldots\right\},\oh_{d_n}\right\}\right|_{u_k=\delta_{k,1}},
$$
where we denote $u_k:=\d_x^k u$.

\medskip

After the substitution $u=\sum_{n\in\mbZ}p_n e^{inx}$ (and therefore $u_k=\sum_{n\in\mbZ}(in)^k p_n e^{inx}$), the Hamiltonians $\oh_d$ can be considered as elements of the algebra 
$$
\hcB:=\mbC[p_1,p_2,\ldots][[p_0,p_{-1},\ldots,\eps]],
$$
with the Poisson bracket given by $\{p_a,p_b\}:=ia\delta_{a+b,0}$. The Poisson algebra $(\hcB,\{\cdot,\cdot\})$ admits a standard deformation quantization $(\hcB[[\hbar]],\star)$, where $\star$ is the Moyal product with the commutation relation $[p_a,p_b]=p_a\star p_b-p_b\star p_a=ia\hbar\delta_{a+b,0}$. In~\cite{BR16}, the authors quantized the KdV hierarchy constructing pairwise commuting elements $\oH_d\in\hcB[[\hbar]]$, called the \emph{quantum Hamiltonians}, satisfying $\oH_d=\oh_d+O(\hbar)$. Moreover, the elements $\oH_d$ have the form $\oH_d=\int H_d dx$, where $H_d$ is a polynomial in $u,u_1,\ldots,\eps,\hbar$:
\begin{align*}
\oH_1&=\int\left(\frac{u^3}{6}+\frac{\eps^2}{24}u u_{xx}-\frac{i\hbar}{24}u\right)dx,\\
\oH_2&=\int\left(\frac{u^4}{24}+\eps^2\frac{u^2 u_{xx}}{48}+\eps^4\frac{u u_{xxxx}}{480}-i\hbar\frac{2uu_{xx}+u^2}{48}-i\hbar\eps^2\frac{u}{2880}\right)dx,\\
&\hspace{2mm}\vdots
\end{align*}

\medskip

By a construction from~\cite{BDGR20}, which was made precise in~\cite{Blo22}, one can associate to the collection of quantum Hamiltonians $\oH_d$ a quantum tau-function $\exp(\mcF^{(q)})$, where $\mcF^{(q)}\in\mbC[[t_0,t_1,\ldots,\eps,\hbar]]$ is uniquely determined by the relations:
\begin{align}
&\left.\frac{\d^{n+1}\mcF^{(q)}}{\d t_0\d t_{d_1}\ldots\d t_{d_n}}\right|_{t_*=0}=\hbar^{1-n}\left.\left[\left[\ldots\left[\left[\frac{\delta\oH_{d_1}}{\delta u},\oH_{d_2}\right],\oH_{d_3}\right],\ldots\right],\oH_{d_n}\right]\right|_{u_k=\delta_{k,1}},\label{eq:construction1}\\
&\frac{\d\mcF^{(q)}}{\d t_0}=\sum_{i\ge 0}t_{i+1}\frac{\d\mcF^{(q)}}{\d t_i}+\frac{t_0^2}{2}-\frac{i\hbar}{24},\label{eq:construction2}
\end{align}
together with the explicit formula for the constant term $\left.\mcF^{(q)}\right|_{t_*=0}=-\frac{i}{5760}\eps^2\hbar$. Note that $\left.\mcF^{(q)}\right|_{\hbar=0}=\mcF$. Equation~\eqref{eq:construction2} is called the \emph{quantum string equation}. Note that it is a close analog of the classical string equation for the formal power series~$\mcF$.

\medskip

\emph{Quantum intersection numbers} $\<\tau_{d_1}\ldots\tau_{d_n}\>_{l,g-l}$ are defined by
$$
\<\tau_{d_1}\ldots\tau_{d_n}\>_{l,g-l}:=i^{\sum d_j-3g-n+3}\Coef_{\eps^{2l}\hbar^{g-l}}\left.\frac{\d^n\mcF^{(q)}}{\d t_{d_1}\ldots\d t_{d_n}}\right|_{t_*=0}.
$$
Note that they are nonzero only if $2g-2+n>0$.

\medskip

\begin{remark}
Our normalization of the quantum intersection numbers is different from the one from the paper~\cite{Blo22}: the coefficient $i^{\sum d_j-3g-n+3}$ is replaced by $i^{g-l}$ there. In~\cite{Blo22}, the author proved that his normalization gives a rational number. The two normalizations are different by the multiplication by $i^{\sum d_j-(n-l-3)}$. In~\cite{Blo22}, the author proved that a quantum intersection number vanishes unless $\sum d_j=n-l+1\text{ mod 2}$. So both normalizations give rational numbers, which coincide or differ by sign. We prefer our normalization, because then there is no sign difference in the formula relating the quantum intersection numbers with the relative Gromov--Witten invariants of $\CP^1$ in the theorem below.
\end{remark}

\medskip

Denote $S(z):=\frac{e^{z/2}-e^{-z/2}}{z}\in\mbQ[[z]]$. In~\cite{Blo22}, the first author proved that
\begin{gather}\label{eq:main formula}
\sum_{d_1,\ldots,d_n\ge 0}\<\tau_{d_1}\ldots\tau_{d_n}\>_{0,g}\mu_1^{d_1}\cdots\mu_n^{d_n}=\left(\sum\mu_j\right)^{2g-3+n}\Coef_{z^{2g}}\left(\frac{\prod_{j=1}^n S(\mu_j z)}{S(z)}\right),
\end{gather}
for $g\ge 0$ and $n\ge 1+2\delta_{g,0}$.

\medskip

\begin{remark}
In~\cite{Blo22}, this theorem was interpreted in the following way. For two tuples $\mu=(\mu_1,\ldots,\mu_k)\in\mbZ_{\ge 1}^k$ and $\nu=(\nu_1,\ldots,\nu_m)\in\mbZ_{\ge 1}^m$, $k,m\ge 1$, with $\sum\mu_i=\sum\nu_j$, denote by~$H^g_{\mu,\nu}$ the double Hurwitz number (we assume that all the points in the preimages of $0$ and~$\infty$ in the ramified coverings that we count are marked). In~\cite{GJV05}, the authors proved that
$$
H^g_{\sum\mu_j,(\mu_1,\ldots,\mu_n)}=r!\left(\sum\mu_j\right)^{r-1}\Coef_{z^{2g}}\left(\frac{\prod_{j=1}^n S(\mu_j z)}{S(z)}\right),\quad n\ge 1,
$$
where $r:=2g-1+n$. Therefore, formula~\eqref{eq:main formula} can be equivalently written as
$$
\<\tau_{d_1}\ldots\tau_{d_n}\>_{0,g}=\Coef_{\mu_1^{d_1}\cdots\mu_n^{d_n}}\left(\frac{H^g_{\sum\mu_j,(\mu_1,\ldots,\mu_n)}}{r!\sum\mu_j}\right),\quad g\ge 0,\quad n\ge 1+2\delta_{g,0}.
$$
\end{remark}

\medskip

In our paper, we relate the quantum intersection numbers $\<\tau_0\tau_{d_1}\ldots\tau_{d_n}\>_{l,g-l}$ to the stationary relative Gromov--Witten invariants of $(\CP^1,0,\infty)$ with an insertion of a Hodge class. For two tuples $\mu=(\mu_1,\ldots,\mu_k)\in\mbZ_{\ge 1}^k$ and $\nu=(\nu_1,\ldots,\nu_m)\in\mbZ_{\ge 1}^m$, $k,m\ge 0$, with $\sum\mu_i=\sum\nu_j$ (we allow the case $\mu=\nu=\emptyset$) denote by $\oM_{g,n}^{\circ}(\CP^1,\mu,\nu)$ the moduli space of stable relative maps from genus $g$, $n$-pointed connected curves to $(\CP^1,0,\infty)$, with ramification profiles over~$0$ and~$\infty$ given by the tuples $\mu$ and $\nu$, respectively. In our definition, we label the points in each ramification profile. This moduli space was defined in the algebro-geometric setting in \cite{Li01}. 

\medskip

The moduli space~$\oM_{g,n}^\circ(\CP^1,\mu,\nu)$ is endowed with
\begin{itemize}
\item \emph{psi-classes} $\psi_i\in H^2(\oM_{g,n}^\circ(\CP^1,\mu,\nu),\mbQ)$, $1\le i\le n$, that are defined as the first Chern classes of the cotangent line bundles over $\oM_{g,n}^\circ(\CP^1,\mu,\nu)$;

\smallskip

\item \emph{evaluation maps} $\ev_i\colon \oM_{g,n}(\CP^1,\mu,\nu)\to\CP^1$, $1\le i\le n$;

\smallskip

\item \emph{Hodge classes} $\lambda_i:=c_i(\mbE)\in H^{2i}(\oM_{g,n}(\CP^1,\mu,\nu),\mbQ)$, where $\mbE$ is the rank $g$ vector bundle over $\oM_{g,n}(\CP^1,\mu,\nu)$ whose fibers are the spaces of holomorphic differentials on nodal curves;

\smallskip

\item \emph{virtual fundamental class} $[\oM_{g,n}(\CP^1,\mu,\nu)]^\virt\in H_{2(2g-2+k+m+n)}(\oM_{g,n}(\CP^1,\mu,\nu),\mbQ)$.
\end{itemize}

\medskip

Let $\omega\in H^{2}\left(\mathbb{CP}^{1},\mbQ\right)$ be the Poincar\'e dual of a point. We denote by
\begin{equation}
\left\langle \mu,\lambda_{l}\prod_{i=1}^{n}\tau_{d_{i}}(\omega),\nu\right\rangle _{g}^{\circ},\quad g,n,l\ge0,\quad d_{1},\ldots,d_{n}\ge0,\label{eq: connected Hodge GW invariants}
\end{equation}
the invariant given by
\[
\int_{\left[\oM_{g,n}^{\circ}(\CP^{1},\mu,\nu)\right]^{\virt}}\lambda_{l}\prod_{j=1}^{n}\psi_{j}^{d_{j}}\ev_{j}^{*}(\omega).
\]
Note that the invariant (\ref{eq: connected Hodge GW invariants}) is zero unless $2g-2+l(\mu)+l(\nu)=\sum d_{i}+l$. Therefore, the genus can be omitted in the notation. 

\medskip

Let $\oM_{g,n}^{\bullet}(\CP^{1},\mu,\nu)$ be the moduli space of stable relative maps with possibly disconnected domains. The brackets $\left\langle \,\,\right\rangle ^{\bullet}$ will be used for the integration over the moduli space of stable relative maps with possibly disconnected domains. 

\medskip

\begin{theorem}
\label{theorem:quantum and GW}
Let $g,l\ge 0$, $n\ge1$, and $\od=(d_{1},\ldots,d_{n})\in\mathbb{Z}_{\ge0}^{n}$.
\begin{enumerate}
\item For any $k\ge 1$, the integral
\[
P_{g,l,\od}(a_{1},\ldots,a_{k}):=\left\langle A,\lambda_{l}\prod_{j=1}^{n}\tau_{d_{i}}(\omega),(a_{1},\ldots,a_{k})\right\rangle ^{\circ},\quad a_{1},\ldots,a_{k}\in\mathbb{Z}_{\ge1},\quad A=\sum a_i,
\]
is a polynomial in $a_{1},\ldots,a_{k}$ of degree $2g+n-1$ with the parity of $n-1$. 
 
\smallskip 
 
\item Let $k:=\sum d_{j}-2g+l+1$. Then
\[
\left\langle \tau_{0}\tau_{d_{1}}\ldots\tau_{d_{n}}\right\rangle {}_{l,g-l}=\begin{cases}
\frac{1}{k!}\Coef_{a_{1}\cdots a_{k}}P_{g,l,\od}, & \text{if \ensuremath{k\ge1}},\\
(-1)^{g}\int_{\oM_{g,2}}\lambda_{g}\lambda_{l}\psi_{1}^{d_{1}}, & \text{if \ensuremath{k=0} and \ensuremath{n=1}},\\
0, & \text{otherwise}.
\end{cases}
\]
\end{enumerate}
\end{theorem}

\medskip

In the case $l=0$, the invariant~\eqref{eq: connected Hodge GW invariants} is a classical stationary relative Gromov--Witten invariant of $\CP^1$, for which Okounkov and Pandharipande~\cite{OP06} found an explicit formula using the infinite wedge formalism. Combining this formula with our theorem, we give a new, much shorter, proof of formula~\eqref{eq:main formula}.

\medskip

\subsection{Plan of the paper}

In Section~\ref{sec: quantum numbers}, we recall the necessary background on quantum intersection numbers. In Section~\ref{sec: proof main}, we prove Theorem~\ref{theorem:quantum and GW}. Finally, in Section~\ref{sec: proof Hurwitz} we give a new proof of formula~(\ref{eq:main formula}) using Theorem~\ref{theorem:quantum and GW}.

\medskip

\subsection{Notations and conventions}

\begin{itemize}

\item For $n\in\mbZ_{\ge 0}$, we denote $[n]:=\{1,\ldots,n\}$.

\smallskip

\item Given a tuple of numbers $(a_1,\ldots,a_n)$, we denote by the capital letter $A$ the sum $A:=\sum_{i=1}^n a_i$. For $I\subset[n]$, we denote $A_I:=\sum_{i\in I}a_i$, $I^c:=[n]\backslash I$, and $a_I:=(a_{i_1},\ldots,a_{i_{|I|}})$, where $\{i_1,\ldots,i_{|I|}\}$, $i_1<\ldots<i_{|I|}$.

\smallskip

\item For a topological space $X$, we denote by $H^i(X)$ the cohomology groups with the coefficients in $\mbQ$. 

\smallskip

\item The moduli space $\oM_{g,n}$ is empty unless the condition $2g-2+n>0$ is satisfied. We will often omit mentioning this condition explicitly, and silently assume that it is satisfied when a moduli space is considered. 

\end{itemize}

\medskip

\subsection{Acknowledgments}
This work started at the conference Geometry and Integrability in Moscow in September 2023. 

\medskip

X. B. was supported by the Netherlands Organization for Scientific Research. The work of A.~B. is an output of a research project implemented as part of the Basic Research Program at the National Research University Higher School of Economics (HSE University). 

\medskip


\section{Quantum intersection numbers}\label{sec: quantum numbers}

\subsection{Hamiltonian structure of the KdV hierarchy}

Let us briefly recall main notions and notations in the formal theory of evolutionary PDEs with one spatial variable (and with one dependent variable):
\begin{itemize}
\item We consider formal variables $u_d$ with $d\ge 0$ and introduce the algebra of \emph{differential polynomials} $\hcA:=\mbC[[u_0]][u_{\ge 1}][[\eps]]$. We denote $u:=u_0$, $u_x:=u_1$, $u_{xx}:=u_2$, \ldots. Denote by $\hcA_d\subset\hcA$ the homogeneous component of (differential) degree $d$, where $\deg u_i:=i$ and $\deg\eps:=-1$. 

\smallskip

\item An operator $\d_x\colon\hcA\to\hcA$ is defined by $\d_x:=\sum_{d\ge 0}u_{d+1}\frac{\d}{\d u_d}$.

\smallskip

\item The operator of \emph{variational derivative} $\frac{\delta}{\delta u}\colon\hcA\to\hcA$ is defined by $\frac{\delta}{\delta u}:=\sum_{i\ge 0}(-\d_x)^i\circ\frac{\d}{\d u_i}$.

\smallskip

\item The space of \emph{local functionals} is defined by $\hLambda:=\hcA/(\Im(\d_x)\oplus\mbC[[\eps]])$. The image of $f\in\hcA$ under the canonical projection $\hcA\to\hLambda$ is denoted by $\int f dx$. The grading on $\hcA$ induces a grading on $\hLambda$. We denote by $\hLambda_d\subset\hLambda$ the homogeneous component of degree~$d$.

\smallskip

\item The kernel of the operator $\frac{\delta}{\delta u}\colon\hcA\to\hcA$ is equal to $\Im(\d_x)\oplus\mbC[[\eps]]$, so the operator $\frac{\delta}{\delta u}\colon\hLambda\to\hcA$ is well defined.

\smallskip

\item The space $\hLambda$ is endowed with a Poisson bracket given by $\{\oh,\og\}:=\int\frac{\delta\oh}{\delta u}\d_x\frac{\delta\og}{\delta u} dx$, where $\oh,\og\in\hLambda$. This bracket has a lifting to a bracket $\{\cdot,\cdot\}\colon\hcA\times\hLambda\to\hcA$ defined by $\{f,\og\}:=\sum_{n\ge 0}\frac{\d f}{\d u_n}\d_x^{n+1}\frac{\delta\og}{\delta u}$.
\end{itemize}

\medskip

The KdV hierarchy is Hamiltonian, $\frac{\d u}{\d t_n}=\{u,\oh_n\}$, where the Hamiltonians $\oh_n$ can be described using the Lax formalism: $\oh_d=\frac{\eps^{2d+4}}{(2d+3)!!}\int\res L^{\frac{2d+3}{2}}dx$, where $L=\d_x^2+2\eps^{-2}u$.

\medskip

As we already discussed in the introduction, the Hamiltonians of the KdV hierarchy can be considered as elements of the other space $\hcB=\mbC[p_1,p_2,\ldots][[p_0,p_{-1},\ldots,\eps]]$. For this, we define a linear map $\phi\colon\hcA\to\hcB[[e^{ix},e^{-ix}]]$ by
$$
\phi(f):=f\big|_{u_d\mapsto\sum_{a\in\mbZ}(ia)^d p_a e^{iax}}\in\hcB[[e^{ix},e^{-ix}]],\quad f\in\hcA.
$$
The map $\phi$ is an injection. Consider the decomposition 
$$
\phi(f)=\sum_{a\in\mbZ}\phi_a(f)e^{iax},\quad \phi_a(f)\in\hcB.
$$
Denote $\tphi_0(f):=\phi_0(f)-\left.\phi_0(f)\right|_{p_*=0}$. We have $\Im(\d_x)\oplus\mbC[[\eps]]\subset\Ker(\tphi_0)$. Therefore, the map $\tphi_0\colon\hLambda\to\hcB$ is well defined. This linear map is injective, and moreover it is compatible with the Poisson structures on $\hLambda$ and $\hcB$,
$$
\tphi_0\left(\{\of,\oh\}\right)=\left\{\tphi_0(\of),\tphi_0(\oh)\right\},\quad \of,\oh\in\hLambda,
$$
where the Poisson structure on $\hcB$ is defined by
$$
\{\of,\oh\}:=\sum_{a\in\mbZ}ia\frac{\d\of}{\d p_a}\frac{\d\oh}{\d p_{-a}},\quad \of,\oh\in\hcB.
$$
Abusing notation, we will denote a KdV Hamiltonian $\oh_d\in\hLambda$ and its image $\tphi_0(\oh_d)\in\hcB$ by the same letter $\oh_d$. 

\medskip

\subsection{The Hamiltonians of the KdV hierarchy and the double ramification cycles}

Consider an $n$-tuple of integers $(a_1,\ldots,a_n)$ such that $\sum a_i=0$. Let us briefly recall the definition of the \emph{double ramification (DR) cycle} $\DR_g(a_1,\ldots,a_n)\in H^{2g}(\oM_{g,n})$. The positive parts of $(a_1,\ldots,a_n)$ define a partition $\mu=(\mu_1,\ldots,\mu_{l(\mu)})$. The negatives of the negative parts of $(a_1,\ldots,a_n)$ define a second partition $\nu=(\nu_1,\ldots,\nu_{l(\nu)})$. Since the parts of $(a_1,\ldots,a_n)$ sum to~$0$, we have $|\mu|=|\nu|$. We allow the case $|\mu|=|\nu|=0$. Let $n_0:=n-l(\mu)-l(\nu)$. The moduli space 
$$
\oM_{g,n_0}^\sim(\CP^1,\mu,\nu)
$$
parameterizes stable relative maps of connected algebraic curves of genus $g$ to rubber $\CP^1$ with ramification profiles $\mu,\nu$ over the points $0,\infty\in\CP^1$, respectively. There is a natural map 
$$
\st\colon \oM_{g,n_0}^\sim(\CP^1,\mu,\nu)\to\oM_{g,n}
$$
forgetting everything except the marked domain curve. The moduli space $\oM_{g,n_0}^\sim(\CP^1,\mu,\nu)$ possesses a virtual fundamental class $\left[\oM_{g,n_0}^\sim(\CP^1,\mu,\nu)\right]^\virt$, which is a homology class of degree $2(2g-3+n)$. The {\it double ramification cycle} 
$$
\DR_g(a_1,\ldots,a_n)\in H^{2g}(\oM_{g,n})
$$
is defined as the Poincar\'e dual to the push-forward $\st_*\left[\oM_{g,n_0}^\sim(\CP^1,\mu,\nu)\right]^\virt\in H_{2(2g-3+n)}(\oM_{g,n})$.

\medskip

The crucial property of the DR cycle is that for any cohomology class $\theta\in H^*(\oM_{g,n})$ the integral $\int_{\oM_{g,n+1}}\lambda_g\DR_g\left(-\sum a_i,a_1,\ldots,a_n\right)\theta$ is a homogeneous polynomial in $a_1,\ldots,a_n$ of degree~$2g$ \cite[Lemma~3.2]{Bur15}. 

\medskip

An explicit formula for the KdV Hamiltonians $\oh_d$ in terms of the geometry of $\oM_{g,n}$ was found in~\cite[Section~4.3.1]{Bur15}. For any $d\in\mbZ_{\ge 0}$, define
$$
h_d:=\sum_{g\ge 0,\,n\ge 2}\frac{\eps^{2g}}{n!}\sum_{\substack{d_1,\ldots,d_n\in\mbZ_{\ge 0}\\\sum d_i=2g}}\Coef_{(a_1)^{d_1}\cdots (a_n)^{d_n}}\left(\int_{\oM_{g,n+1}}\lambda_g\psi_1^d\DR_g\left(-\sum a_i,a_1,\ldots,a_n\right)\right)u_{d_1}\cdots u_{d_n}.
$$
Then these differential polynomials give densities for the KdV Hamiltonians: $\oh_d=\int h_d dx$. In other language, as elements of $\hcB$, the Hamiltonians $\oh_d$ are given by
$$
\oh_d=\sum_{g\ge 0,\,n\ge 2}\frac{(-\eps^2)^g}{n!}\sum_{\substack{a_1,\ldots,a_n\in\mbZ\\\sum a_i=0}}\left(\int_{\oM_{g,n+1}}\lambda_g\psi_1^d\DR_g\left(0,a_1,\ldots,a_n\right)\right)p_{a_1}\cdots p_{a_n}.
$$

\medskip

\subsection{Quantum KdV hierarchy and quantum intersection numbers}

Consider the vector space $\hcB[[\hbar]]$. The \emph{Moyal product} $\star$ on it is defined by
$$
f \star h=e^{\sum_{k>0} i \hbar k \frac{\d}{\d p_{k}}\frac{\d}{\d q_{-k}}}\left.(f(p_*,\eps,\hbar)h(q_*,\eps,\hbar))\right|_{q_c\mapsto p_c},\quad f,h\in\hcB[[\hbar]],
$$
where $q_k$, $k\in\mbZ$, are additional formal variables and $h(q_*,\eps,\hbar):=h|_{p_a\mapsto q_a}$. The resulting algebra structure on $\hcB[[\hbar]]$ is a deformation quantization of the Poisson algebra $(\hcB,\{\cdot,\cdot\})$ in the sense that for $f=\sum_{i\ge 0}f_i\hbar^i$ and $h=\sum_{i\ge 0}h_i\hbar^i$, $f_i,h_i\in\hcB$, we have $[f,h]=\hbar\{f_0,h_0\}+O(\hbar^2)$.

\medskip

We extend the linear maps $\phi\colon\hcA\to\hcB[[e^{ix},e^{-ix}]]$ and $\tphi_0\colon\hLambda\to\hcB$ to linear maps $\phi\colon\hcA[[\hbar]]\to\hcB[[\hbar]][[e^{ix},e^{-ix}]]$ and $\tphi_0\colon\hLambda[[\hbar]]\to\hcB[[\hbar]]$ by coefficient-wise action. Note that if $f\in\Im(\phi)\subset\hcB[[\hbar]][[e^{ix},e^{-ix}]]$ and $\oh\in\Im(\tphi_0)\subset\hcB[[\hbar]]$, then $[f,\oh]\in\Im(\phi)$~\cite[formula~(2.2)]{BR16}. This implies that the subspace $\Im(\tphi_0)\subset\hcB[[\hbar]]$ is closed under the commutator $[\cdot,\cdot]$.

\medskip

In~\cite{BR16}, the authors defined elements $\oH_d\in\hcB[[\hbar]]$ by
$$
\oH_d:=\sum_{g\ge 0,\,n\ge 1}\frac{1}{n!}\sum_{\substack{a_1,\ldots,a_n\in\mbZ\\\sum a_i=0}}\left(\int_{\oM_{g,n+1}}\left(\sum_{j=0}^g(i\hbar)^{g-j}(-\eps^2)^j\lambda_j\right)\psi_1^d\DR_g\left(0,a_1,\ldots,a_n\right)\right)p_{a_1}\cdots p_{a_n}.
$$
Clearly, $\oH_d=\oh_d+O(\hbar)$. In~\cite[Theorem~3.4]{BR16}, the authors proved that $[\oH_{d_1},\oH_{d_2}]=0$, and so the elements $\oH_d$ give a quantization of the KdV hierarchy. As we already mentioned in the introduction, the elements $\oH_d$ are called the \emph{quantum Hamiltonians}.

\medskip
 
Since the integral $\int_{\oM_{g,n+1}}\lambda_j\psi_1^d\DR_g\left(0,a_1,\ldots,a_n\right)$ is a polynomial in $a_1,\ldots,a_n$~\cite[Proposition~B.1]{BR16}, the elements $\oH_d\in\hcB[[\hbar]]$ belong to the image of the inclusion $\tphi_0\colon\hLambda[[\hbar]]\to\hcB[[\hbar]]$. Therefore, 
\begin{gather*}
\hbar^{1-n}\left[\left[\ldots\left[\left[\frac{\delta\oH_{d_1}}{\delta u},\oH_{d_2}\right],\oH_{d_3}\right],\ldots\right],\oH_{d_n}\right]\in\hcA[[\hbar]],\quad n\ge 1,\,d_1,\ldots,d_n\ge 0.
\end{gather*}
In~\cite{BDGR20}, the authors checked that the multiple commutator here is symmetric with respect to all permutations of the numbers $d_1,\ldots,d_n$. In~\cite{Blo22}, the author proved that there exists a formal power series $\mcF^{(q)}\in\mbC[[t_0,t_1,\ldots,\eps,\hbar]]$ satisfying equations~\eqref{eq:construction1} and~\eqref{eq:construction2}, and moreover a solution is unique up to the constant term $\left.\mcF^{(q)}\right|_{t_*=0}$. This constant term is chosen in such a way that the equation
\begin{gather}\label{eq:quantum dilaton}
\frac{\d\mcF^{(q)}}{\d t_1}=\left(\sum_{n\ge 0}t_n\frac{\d}{\d t_n}+\eps\frac{\d}{\d\eps}+2\hbar\frac{\d}{\d\hbar}-2\right)\mcF^{(q)}+\frac{\eps^2}{24},
\end{gather}
called the~\emph{quantum dilaton equation}, is satisfied after the substitution $t_*=0$. Since 
$$
\<\tau_1\>_{l,g-l}=
\begin{cases}
\frac{1}{24},&\text{if $g=l=1$},\\
\frac{1}{2880},&\text{if $g=2$ and $l=1$},\\
0,&\text{otherwise},
\end{cases}
$$
this gives the constant term $\left.\mcF^{(q)}\right|_{t_*=0}=-\frac{i}{5760}\eps^2\hbar$. In~\cite{Blo22}, the author conjectured that equation~\eqref{eq:quantum dilaton} is true: it is clearly compatible with the classical dilaton equation for the formal power series~$\mcF$, and in~\cite{Blo22} the author also showed that it is true after the substitution $\eps=0$, however the quantum dilaton equation is not fully proved yet.

\medskip


\section{Proof of Theorem~\ref{theorem:quantum and GW}}\label{sec: proof main}

We will establish using the degeneration formula that
\begin{align}
{\rm Coef}_{\eps^{2l}\hbar^{g-l+n-1}}&\left.\bigg[\bigg[\ldots \left.\left.\left[\left[\frac{\delta\overline{H}_{d_{1}}}{\delta u},\overline{H}_{d_{2}}\right],\overline{H}_{d_{3}}\right],\dots\right],\overline{H}_{d_{n}}\right]\right|_{p_{\le 0}=0}=\label{eq: GW and DR}\\
 =&\delta_{k\ge 1}\sum_{a_{1},\dots,a_{k}>0}i^{g+l+n-1}\left\langle A,\lambda_{l}\prod_{i=1}^{n}\tau_{d_{i}}(\omega),\left(a_{1},\dots,a_{k}\right)\right\rangle ^{\circ}\frac{p_{a_{1}}\cdots p_{a_{k}}}{k!}e^{iAx}\nonumber \\
 &+\delta_{k,0}\delta_{n,1}i^{3g+l}\int_{\overline{\mathcal{M}}_{g,2}}\lambda_{g}\lambda_{l}\psi_1^{d_{1}},\nonumber 
\end{align}
where $k:=\sum d_{j}-2g+l+1$.

\medskip

First, let us denote $H_{d-1}:=\frac{\delta\overline{H}_{d}}{\delta u}$, $d\ge 0$. We have
\[
H_{d-1}=\sum_{g,m\geq 0}\frac{\left(i\hbar\right)^{g}}{m!}\sum_{a_{1},\dots,a_{m}\in\mathbb{Z}}\left(\int_{\oM_{g,m+2}}{\rm DR}_{g}\left(0,a_{1},\dots,a_{m},-A\right)\psi_{1}^d\Lambda\left(\frac{-\eps^{2}}{i\hbar}\right)\right)p_{a_{1}}\cdots p_{a_{m}}e^{iAx},
\]
where $\Lambda(s):=1+s\lambda_{1}+\cdots+s^{g}\lambda_{g}$. 

Then, we introduce the non associative product $f\tilde{\star}g:=f\star g-fg$. This slight modification of the star product is convenient for the following reason: the coefficient of $p_{a_{1}}\cdots p_{a_{m}}e^{Aix}\eps^{2l}\hbar^{g-l+n-1}$ in both $\left[\left[\ldots\left[\left[H_{d_{1}-1},\overline{H}_{d_{2}}\right],\overline{H}_{d_{3}}\right],\dots\right],\overline{H}_{d_{n}}\right]$ and $(\cdots((H_{d_{1}-1}\tilde{\star}\overline{H}_{d_{2}})\tilde{\star} \overline{H}_{d_{3}}) \tilde{\star}\cdots\tilde{\star}\overline{H}_{d_{n}})$ are identical when $a_{1},\dots,a_{m}>0$ and $m\geq0$. This occurs for two reasons: first the bracket $[\cdot,\cdot]$ defined as the commutator of the $\star$-product is also the commutator of the $\tilde{\star}$-product, and second 
\[
{\rm Coef}_{p_{a_{1}}\cdots p_{a_{m}}e^{Aix}\eps^{2l}\hbar^{g-l+n-1}}\overline{H}_{d_{n}}\tilde{\star}\left[\ldots\left[H_{d_{1}-1},\overline{H}_{d_{2}}\right],\dots,\overline{H}_{d_{n-1}}\right]=0,\quad a_{1},\dots,a_{m}>0.
\]
This vanishing happens because $\overline{H}_{d_{n}}\tilde{\star}\left(\cdots\right)$ involves a product of derivative $\frac{\partial}{\partial p_{k_1}}\cdots\frac{\partial}{\partial p_{k_s}}$, with $k_1,\dots,k_s>0$, from the star product acting on $\overline{H}_{d_{n}}$ and since we use the $\tilde{\star}$-product we have $s\geq1$. Consequently, when extracting the coefficient of $p_{a_{1}}\cdots p_{a_{m}}e^{Aix}\eps^{2l}$ with the condition $a_{1},\dots,a_{m}>0$, the sum of the parts of each DR-cycle from $\overline{H}_{d_{n}}$ is positive, contradicting the requirement that they sum to zero. This proves the vanishing. Then, a direct induction justifies the statement. Thus, in (\ref{eq: GW and DR}) we equivalently write the LHS with commutators or with $\tilde{\star}$-products.
\medskip

Before establishing (\ref{eq: GW and DR}), we show how this formula proves assertions $1$ and $2$ of Theorem~\ref{theorem:quantum and GW}.

\medskip

\subsection{Proving assertion 1}

The polynomial behavior of $P_{g,l,\od}(a_{1},\ldots,a_{k})$ follows from equation~(\ref{eq: GW and DR}) by the analysis of \cite[Section 6]{Blo22}. More precisely, one first writes an expression for $(\cdots(H_{d_{1}-1}\tilde{\star}\overline{H}_{d_{2}})\tilde{\star}\cdots\tilde{\star}\overline{H}_{d_{n}})$ from the expression of $H_{d_{1}-1}\star\overline{H}_{d_{2}}\star\cdots\star\overline{H}_{d_{n}}$ given by Eq.~(36) in~\cite{Blo22}, this only accounts for adding the conditions $\gamma$ described in the same section. Then, after extracting the coefficient of $p_{a_{1}}\cdots p_{a_{k}}e^{Aix}\eps^{2l}\hbar^{g-l+n-1}$ in this expression of $(\cdots(H_{d_{1}-1}\tilde{\star}\overline{H}_{d_{2}})\tilde{\star}\cdots\tilde{\star}\overline{H}_{d_{n}})$, the polynomality, degree and parity properties follows from \cite[Lemma~6.3]{Blo22}.

\medskip

\subsection{Proving assertion 2}

The correlator $\left\langle \tau_{0}\tau_{d_{1}}\ldots\tau_{d_{n}}\right\rangle {}_{l,g-l}$ is obtained from (\ref{eq: GW and DR}) by evaluating at $u_{i}=\delta_{i,1}$. This evaluation is done using Lemma $6.2$ in \cite{BDGR18} and directly yields the result.

\medskip

\subsection{Proving (\ref{eq: GW and DR})}

By a direct dimension counting, the coefficient of $p_{a_{1}}\cdots p_{a_{m}}e^{Aix}\eps^{2l}\hbar^{g-l+n-1}$ in $(\cdots(H_{d_{1}-1}\tilde{\star}\overline{H}_{d_{2}})\tilde{\star}\cdots\tilde{\star}\overline{H}_{d_{n}})$ can only be non zero if $m=k$.

\medskip

If $k=0$, then $\left.{\rm Coef}_{\eps^{2l}\hbar^{g-l+n-1}}(\cdots(H_{d_{1}-1}\tilde{\star}\overline{H}_{d_{2}})\tilde{\star}\cdots\tilde{\star}\overline{H}_{d_{n}})\right|_{p_{*}=0}$ vanishes for $n>1$ since, once again, the parts of the DR cycles coming from $\overline{H}_{d_{n}}$ after the action of the $\tilde\star$-product and the evaluation $p_*=0$ cannot add up to zero. When $n=1$, we have $\left.{\rm Coef}_{\eps^{2l}\hbar^{g-l+n-1}}H_{d_{1}-1}\right|_{p_{*}=0}=i^{g+l}\int_{\oM_{g,2}}{\rm DR}_{g}(0,0)\psi_1^{d_{1}}\lambda_{l}$ and the property of the DR cycle ${\rm DR}_{g}(0,0)=(-1)^{g}\lambda_{g}$ yields the result.

\medskip

We now prove the main ingredient of (\ref{eq: GW and DR}): when $k>0$, we show that 
\begin{multline}
\sum_{a_1,\ldots,a_k>0}\left\langle A,\lambda_{l}\prod_{i=1}^{n}\tau_{d_{i}}(\omega),\left(a_{1},\dots,a_{k}\right)\right\rangle ^{\circ}\frac{p_{a_1}\cdots p_{a_k}}{k!}e^{iAx}=\\
=\left(-i\right)^{g+l+n-1}\left.{\rm Coef}_{\eps^{2l}\hbar^{g-l+n-1}}(\cdots(H_{d_{1}-1}\tilde{\star}\overline{H}_{d_{2}})\tilde{\star}\cdots\tilde{\star}\overline{H}_{d_{n}})\right|_{p_{\le 0}=0}.\label{eq: application degeneration}
\end{multline}
The proof is done by the induction over $n$ using the degeneration formula and the following lemma.

\medskip

\begin{lemma}
\label{lem: Identification GW and DR}
Fix $g,n\geq0$ such that $2g-1+n>0$. Let $b_{1},\ldots,b_{n}\in\mathbb{Z}$ such that $\sum_{i=1}^{n}b_{i}=0$. We have
\[
\int_{\oM_{g,n+1}}\DR_{g}(0,b_{1},\ldots,b_{n})\psi_1^d\lambda_l=
\int_{\left[\overline{\mathcal{M}}_{g,1+n_{0}}^{\circ}(\CP^1,\mu,\nu)\right]^{{\rm virt}}}\lambda_{l}\psi_{1}^{d}\ev_1^*(\omega),
\]
where $\mu$, $\nu$ are the positive and the negatives of the negative parts of $\left(0,b_{1},\dots,b_{n}\right)$, respectively, and $n_{0}+1$ is the number of parts equal to 0.
\end{lemma}
\begin{proof}
Let $\Psi_{1}\in H^{2}\left(\overline{\mathcal{M}}_{g,n_{0}+1}^{\sim}\left(\mathbb{CP}^{1},\mu,\nu\right)\right)$ be the first Chern class of the cotangent line bundle at the point $1$. The cotangent line bundle at the point $1$ over $\overline{\mathcal{M}}_{g,n_{0}+1}^{\sim}\left(\mathbb{CP}^{1},\mu,\nu\right)$ is identified with the pull-back by the map $\st\colon \overline{\mathcal{M}}_{g,n_{0}+1}^{\sim}\left(\mathbb{CP}^{1},\mu,\nu\right)\to\oM_{g,n+1}$ of the cotangent line bundle at the point $1$ over $\overline{\mathcal{M}}_{g,n+1}$ since the component of the point $1$ is not stabilised by the forgetful map, thus we have $\Psi_{1}=\st^*(\psi_{1})$. Moreover, since the DR cycle is, by definition, the push-forward by $\st$ of the virtual fundamental class of $\overline{\mathcal{M}}_{g,n_{0}+1}^{\sim}\left(\CP^1,\mu,\nu\right)$, and furthermore, since $\lambda_i=\st^*(\lambda_i)$, we get by the projection formula
\[
\int_{\oM_{g,n+1}}\DR_{g}(0,b_{1},\ldots,b_{n})\psi_1^d\lambda_l=\int_{\left[\overline{\mathcal{M}}_{g,n_{0}+1}^{\sim}\left(\mathbb{CP}^{1},\mu,\nu\right)\right]^{{\rm virt}}}\Psi_{1}^{d}\lambda_{l}.
\]
Let $p\colon\oM^\circ_{g,n_0+1}(\CP^1,\mu,\nu)\to\oM^\sim_{g,n_0+1}(\CP^1,\mu,\nu)$ be the canonical forgetful map. By \cite[Lemma~2]{MP06}, we have $\left[\oM^\sim_{g,n_0+1}(\CP^1,\mu,\nu)\right]^\virt=p_*\left(\ev_1^*(\omega)\cap\left[\oM^\circ_{g,n_0+1}(\CP^1,\mu,\nu)\right]^\virt\right)$, which by the projection formula yields the result.
\end{proof}

\medskip

The case $n=1$ in (\ref{eq: application degeneration}) directly follows from the statement of the lemma. 

\medskip

We now prove (\ref{eq: application degeneration}) for $n\geq 2$. The degeneration formula \cite[Theorem~3.15]{Li02} gives
\begin{align}
&\left<A,\lambda_{l}\prod_{j=1}^n\tau_{d_j}(\omega),\left(a_{1},\dots,a_{k}\right)\right>^{\bullet}=\label{eq: degeneration disconnected}\\
 & =\sum_{l_{1}+l_{2}=l}\sum_{p\geq1}\sum_{\substack{\mu=(\mu_{1},\ldots,\mu_{p})\in\mathbb{Z}_{\ge1}^p\\
\ensuremath{\mu_{1}+\cdots+\mu_{p}=A}
}
}\frac{\mu_{1}\cdots\mu_{p}}{p!}\left\langle A,\lambda_{l_{1}}\prod_{j=1}^{n-1}\tau_{d_{j}}(\omega),\mu\right\rangle ^{\bullet}\left\langle \mu,\lambda_{l_{2}}\tau_{d_{n}}(\omega),\left(a_{1},\dots,a_{k}\right)\right\rangle ^{\bullet},\nonumber 
\end{align}
where we considered that the target $\mathbb{CP}^{1}$ degenerates to $\mathbb{CP}^{1}\cup\mathbb{CP}^{1}$ intersecting at a node, such that the first $\mathbb{CP}^{1}$ contains the relative point associated to the total ramification $A$ and the images of the points $1,\dots,n-1$, and the second $\mathbb{CP}^{1}$ contains the second relative point and the image of the point $n$. Note that, as in \cite{Li02}, we label the points in the ramification profiles. We now explain how the disconnected contributions on each side of (\ref{eq: degeneration disconnected}) compensate to give a formula involving only connected invariants.

\medskip

\paragraph{\underline{LHS of (\ref{eq: degeneration disconnected})}} 

The term $\left<A,\lambda_{l}\prod_{j=1}^n\tau_{d_j}(\omega),\left(a_{1},\dots,a_{k}\right)\right>^{\bullet}$ on the left-hand side (LHS) of (\ref{eq: degeneration disconnected}) equals the connected contribution $\left<A,\lambda_{l}\prod_{j=1}^n\tau_{d_j}(\omega),\left(a_{1},\dots,a_{k}\right)\right>^{\circ}$, plus contributions with disconnected domains. Since the preimage of $0$ is given by a unique point with total ramification~$A$, the disconnected contributions correspond to whenever the point $1$, or the point $2$, $\dots$, or the point $n$ is on a component of degree $0$. Furthermore, since $\omega^{2}=0$, if the two marked points are on the same component of degree $0$, this contribution vanishes. We get
\begin{align}
&\left<A,\lambda_{l}\prod_{j=1}^n\tau_{d_j}(\omega),\left(a_{1},\dots,a_{k}\right)\right>^{\bullet}=\left<A,\lambda_{l}\prod_{j=1}^n\tau_{d_j}(\omega),\left(a_{1},\dots,a_{k}\right)\right>^{\circ}\label{eq:disconnected from connected}\\
&\hspace{1.5cm}  +\sum_{I\subsetneq\left[n\right]}\sum_{\substack{m\colon\{0\}\sqcup I\to\mbZ_{\ge 0}\\m(0)+\sum_{i\in I}m(i)=l}}\left\langle A,\lambda_{m(0)}\prod_{j\in I^{c}}\tau_{d_{j}}(\omega),\left(a_{1},\dots,a_{k}\right)\right\rangle ^{\circ}\prod_{i\in I}\left<\emptyset,\lambda_{m(i)}\tau_{d_{i}}(\omega),\emptyset\right>^{\circ}\notag\\
&\hspace{1.5cm}  +\delta_{k,1}\underset{=1/A}{\underbrace{\left\langle A,1,A\right\rangle ^{\circ}}}\sum_{m_{1}+\cdots+m_{n}=l}\prod_{i=1}^{n}\left\langle\emptyset,\lambda_{m_i}\tau_{d_{i}}(\omega),\emptyset\right\rangle ^{\circ},\notag
\end{align}
where $I^{c}:=[n]\backslash I$, moreover we used in the last term that $\left\langle A,\lambda_{l},\left(a_{1},\dots,a_{k}\right)\right\rangle ^{\circ}\neq0$ only for $l=0$ and $k=1$ by dimension counting. 

\medskip

\paragraph{\underline{RHS of (\ref{eq: degeneration disconnected})}}

Similarly, the factor $\left\langle A,\lambda_{l_{1}}\prod_{j=1}^{n-1}\tau_{d_j}(\omega),\left(\mu_{1},\dots,\mu_{p}\right)\right\rangle ^{\bullet}$ on the right-hand side (RHS) of equation~(\ref{eq: degeneration disconnected}) is equal to the connected contribution $\left\langle A,\lambda_{l_{1}}\prod_{j=1}^{n-1}\tau_{d_j}(\omega),\left(\mu_{1},\dots,\mu_{p}\right)\right\rangle ^{\circ}$ plus disconnected contributions such that the point $1$, or $2$, $\dots$, or $n-1$ lies on a degree $0$ component, we get
\begin{align*}
&\langle A,\lambda_{l_{1}}\prod_{j=1}^{n-1}\tau_{d_j}(\omega),\left(\mu_{1},\dots,\mu_{p}\right)\rangle^{\bullet}=\left\langle A,\lambda_{l_{1}}\prod_{j=1}^{n-1}\tau_{d_j}(\omega),\left(\mu_{1},\dots,\mu_{p}\right)\right\rangle ^{\circ}\\
&\hspace{2.5cm} +\sum_{I\subsetneq\left[n-1\right]}\sum_{\substack{m\colon\{0\}\sqcup I\to\mbZ_{\ge 0}\\m(0)+\sum_{i\in I}m(i)=l_1}}\left\langle A,\lambda_{m(0)}\prod_{j\in I^{c}}\tau_{d_{j}}(\omega),\left(\mu_1,\ldots,\mu_p\right)\right\rangle ^{\circ}\prod_{i\in I}\left\langle\emptyset,\lambda_{m(i)}\tau_{d_{i}}(\omega),\emptyset\right\rangle^{\circ}\\
&\hspace{2.5cm}+\delta_{p,1}\left\langle A,1,A\right\rangle ^{\circ}\sum_{m_{1}+\cdots+m_{n}=l_1}\prod_{i=1}^{n}\left\langle\emptyset,\lambda_{m_i}\tau_{d_{i}}(\omega),\emptyset\right\rangle^{\circ}.
\end{align*}
The factor $\left\langle \left(\mu_{1},\dots,\mu_{p}\right),\lambda_{l_{2}}\tau_{d_{n}}(\omega),\left(a_{1},\dots,a_{k}\right)\right\rangle ^{\bullet}$ is more complicated because we can split the profile above $0$ and above $\infty$ on different domains, we get
\begin{align*}
\left\langle \left(\mu_{1},\dots,\mu_{p}\right),\lambda_{l_{2}}\tau_{d_{n}}(\omega),\left(a_{1},\dots,a_{k}\right)\right\rangle ^{\bullet} =&\sum_{\substack{I_{1}\sqcup I_{2}=\left[p\right]\\
I_{2}\neq\emptyset}}\sum_{\substack{J_{1}\sqcup J_{2}=\left[k\right]\\
J_{2}\neq\emptyset
}
}\left\langle \mu_{I_{1}},1,a_{J_{1}}\right\rangle ^{\bullet}\left\langle \mu_{I_{2}},\lambda_{l_{2}}\tau_{d_{n}}(\omega),a_{J_{2}}\right\rangle ^{\circ}\\
 & +\left\langle \left(\mu_{1},\dots,\mu_{p}\right),1,\left(a_{1},\dots,a_{k}\right)\right\rangle ^{\bullet}\left\langle\emptyset,\lambda_{l_{2}}\tau_{d_{n}}(\omega),\emptyset\right\rangle ^{\circ}.
\end{align*}
Note that a dimension counting forces the $\lambda$-class to entirely lie on the component with the $n$-th marked point. In addition, the term $\left\langle \mu_{I_{1}},1,a_{J_{1}}\right\rangle ^{\bullet}$ is, again by dimension counting, equal to a product of relative connected invariants with genus $0$ domain and $2$ marked points, each of them contributing as $\left\langle a,1,a\right\rangle ^{\circ}=\frac{1}{a}$, $a>0$. Same story for $\left\langle \left(\mu_{1},\dots,\mu_{p}\right),1,\left(a_{1},\dots,a_{k}\right)\right\rangle ^{\bullet}$.

\medskip

Finally, the contributions of the LHS and RHS of (\ref{eq: degeneration disconnected}) such that at least one marked point from $1,\ldots,n$ lies on a component of degree $0$ compensate. After simplifications, one gets
\begin{align}
&\left\langle A,\lambda_{l} \prod_{j=1}^n\tau_{d_j}(\omega),\left(a_{1},\dots,a_{k}\right)\right\rangle^{\circ}=\label{eq: degeneration connected}\\
=&\sum_{\substack{J_{1}\sqcup J_{2}=\left[k\right]\\l_{1}+l_{2}=l\\p\ge 1}}\sum_{\substack{\mu=(\mu_{1},\ldots,\mu_{p})\in\mathbb{Z}_{\ge 1}^p\\
\mu_{1}+\cdots+\mu_{p}=A_{J_2}}}\frac{\mu_{1}\cdots\mu_{p}}{p!}\left\langle A,\lambda_{l_{1}}\prod_{j=1}^{n-1}\tau_{d_j}(\omega),\left(\mu,a_{J_{1}}\right)\right\rangle ^{\circ}\left\langle \mu,\lambda_{l_{2}}\tau_{d_{n}}(\omega),a_{J_{2}}\right\rangle ^{\circ},\nonumber 
\end{align}
where $\left(\mu,a_{J_{1}}\right)$ is the $(p+|J_1|)$-tuple obtained by concatenation from the tuples $\mu$ and $a_{J_1}$. Now, using Lemma~\ref{lem: Identification GW and DR} and the induction hypothesis, we find that after the multiplication by $\frac{p_{a_1}\cdots p_{a_k}}{k!}e^{iAx}$ and summing over $a_1,\ldots,a_k>0$ the RHS of (\ref{eq: degeneration connected}) equals the coefficient of $\eps^{2l}\hbar^{g-l+n-1}$ in the $\tilde{\star}$-product $\left.(\cdots(H_{d_{1}-1}\tilde{\star}\overline{H}_{d_{2}})\tilde{\star}\cdots\tilde{\star}\overline{H}_{d_{n}})\right|_{p_{\le 0}=0}$, after the multipliation by $\left(-i\right)^{g+l+n-1}$. This establishes (\ref{eq: application degeneration}), thereby concluding the proof. 

\medskip


\section{A new proof of formula~\eqref{eq:main formula}}\label{sec: proof Hurwitz}

Using the quantum string equation, we observe that formula~\eqref{eq:main formula} follows from the equality
$$
\sum_{g\ge 0}\sum_{d_1,\ldots,d_n\ge 0}\<\tau_0\tau_{d_1}\ldots\tau_{d_n}\>_{0,g}z^{2g}\mu_1^{d_1}\cdots\mu_n^{d_n}=\left(\sum\mu_j\right)^{n-2}\frac{\prod_{i=1}^n S\left(\mu_i(\sum\mu_j)z\right)}{S\left((\sum\mu_j)z\right)},\quad n\ge 2,
$$
which can be equivalently written as
\begin{gather}\label{eq:equivalent to main formula}
\sum_{g,d_1,\ldots,d_n\ge 0}\<\tau_0\tau_{d_1}\ldots\tau_{d_n}\>_{0,g}t^{\sum d_i-2g+1}\mu_1^{d_1}\cdots\mu_n^{d_n}=t^{n-1}\left(\sum\mu_j\right)^{n-2}\frac{\prod_{i=1}^n S\left(\mu_i(\sum\mu_j)t\right)}{S\left(\sum\mu_j\right)},\,\, n\ge 2.
\end{gather}
For two tuples $\mu=(\mu_1,\ldots,\mu_k)\in\mbZ_{\ge 1}^k$ and $\nu=(\nu_1,\ldots,\nu_m)\in\mbZ_{\ge 1}^m$, $k,m\ge 0$, with $\sum\mu_i=\sum\nu_j$, we introduce the generating series
$$
F^\circ_{\mu,\nu}(z_1,\ldots,z_n):=\sum_{d_1,\ldots,d_n\ge 0}\<\mu,\prod_{i=1}^n\tau_{d_i}(\omega),\nu\>^\circ z_1^{d_1+1}\cdots z_n^{d_n+1}.
$$
Using Theorem~\ref{theorem:quantum and GW}, we see that formula~\eqref{eq:equivalent to main formula} follows from
\begin{gather*}
\frac{1}{k!}\Coef_{a_1\cdots a_k}F^\circ_{A,(a_1,\ldots,a_k)}(z_1,\ldots,z_n)=\left(\prod_{i=1}^n z_i\right)\frac{Z^{n-2}}{S\left(Z\right)}\Coef_{t^{k-n+1}}\left(\prod_{i=1}^n S\left(z_i Zt\right)\right),\,\, k\ge 1,\,n\ge 1,
\end{gather*}
which is equivalent to
\begin{gather}\label{eq:formula1}
\frac{1}{k!}\Coef_{a_1\cdots a_k}F^\circ_{A,(a_1,\ldots,a_k)}(z_1,\ldots,z_n)=\frac{1}{Z\varsigma(Z)}\Coef_{t^{k+1}}\left(\prod_{i=1}^n \varsigma(z_i Zt)\right),\quad k\ge 1,\,n\ge 1,
\end{gather}
where $Z:=\sum_{i=1}^n z_i$ and $\varsigma(z):=z S(z)=e^{z/2}-e^{-z/2}$. 

\medskip

Recall that a relative Gromov--Witten invariant $\<\mu,\prod_{i=1}^n\tau_{d_i}(\omega),\nu\>^\bullet_g$, $g\in\mbZ$, $n\ge 0$, $d_1,\ldots,d_n\ge 0$, is zero unless $2g-2+l(\mu)+l(\nu)=\sum d_i$. We adopt the convention
$$
\<\mu,\tau_{-2}(\omega)^l\tau_{-1}(\omega)^m\prod_{i=1}^n\tau_{d_i}(\omega),\nu\>^\bullet_g:=\delta_{m,0}\<\mu,\prod_{i=1}^n\tau_{d_i}(\omega),\nu\>^\bullet_{g+l}.
$$
So we will consider the numbers 
$$
\<\mu,\prod_{i=1}^n\tau_{k_i}(\omega),\nu\>^\bullet_g,\quad g\in\mbZ,\quad n\ge 0,\quad k_1,\ldots,k_n\ge -2.
$$
Note that it is still true that such a number vanishes unless the condition $2g-2+l(\mu)+l(\nu)=\sum k_i$ is satisfied. So we can still omit the genus in the notation. Introduce the generating series
$$
F^\bullet_{\mu,\nu}(z_1,\ldots,z_n):=\sum_{k_1,\ldots,k_n\ge -2}\<\mu,\prod_{i=1}^n\tau_{k_i}(\omega),\nu\>^\bullet z_1^{k_1+1}\cdots z_n^{k_n+1}.
$$
Note that (see e.g.~\cite[eq.~(0.26)]{OP06})
$$
\sum_{d\ge 0}\<\emptyset,\tau_d(\omega),\emptyset\>^\circ z^{d+1}=\frac{1}{\varsigma(z)}-\frac{1}{z},
$$
which implies that
$$
\sum_{d\ge -2}\<\emptyset,\tau_d(\omega),\emptyset\>^\bullet z^{d+1}=\frac{1}{\varsigma(z)},
$$
and then using~\eqref{eq:disconnected from connected} we obtain
$$
F^\bullet_{A,(a_1,\ldots,a_k)}(z_1,\ldots,z_n)=\sum_{J\subset [n]}\frac{F^\circ_{A,(a_1,\ldots,a_k)}(z_{J^c})}{\prod_{j\in J}\varsigma(z_j)},\quad k\ge 1,\quad a_1,\ldots,a_k\ge 1.
$$

\medskip

Note that
$$
F^\circ_{A,(a_1,\ldots,a_k)}()=
\begin{cases}
\frac{1}{a_1},&\text{if $k=1$},\\
0,&\text{if $k\ge 2$}.
\end{cases}
$$
So formula~\eqref{eq:formula1} implies that
\begin{gather}\label{eq:formula2}
\frac{1}{k!}\Coef_{a_1\cdots a_k}F^\bullet_{A,(a_1,\ldots,a_k)}(z_1,\ldots,z_n)=\sum_{J\subsetneq[n]}\frac{\Coef_{t^{k+1}}\left(\prod_{j\in J^c}\varsigma\left(z_j Z_{J^c}t\right)\right)}{\left(\prod_{j\in J}\varsigma(z_j)\right)Z_{J^c}\varsigma\left(Z_{J^c}\right)},\quad k,n\ge 1.
\end{gather} 
On the other hand, using the induction on $n$, one can easily see that formula~\eqref{eq:formula2} implies formula~\eqref{eq:formula1}. Let us now prove formula~\eqref{eq:formula2}.

\medskip

We will use an explicit formula for $F^\bullet_{A,(a_1,\ldots,a_k)}(z_1,\ldots,z_n)$ obtained in~\cite{OP06} using the infinite wedge formalism (see~\cite[Section~2]{OP06}), which we briefly recall now. We consider a vector space $V$ with basis $\{\underline{k}\}$ indexed by the half-integers:
$$
V=\bigoplus_{k\in\mbZ+\frac{1}{2}}\mbC\underline{k}.
$$
Denote by $\Lambda^{\frac{\infty}{2}}V$ the vector space spanned by the infinite wedge products 
$$
\underline{a_1}\wedge\underline{a_2}\wedge\ldots\quad\text{with $a_i=-i+\frac{1}{2}+c$ for some $c\in\mbZ$ and $i$ big enough.}
$$
For any $k\in\mbZ+\frac{1}{2}$, define linear operators $\psi_k,\psi_k^*\colon\Lambda^{\frac{\infty}{2}}V\to\Lambda^{\frac{\infty}{2}}V$ by 
\begin{align*}
&\psi_k(\underline{a_1}\wedge\underline{a_2}\wedge\ldots):=\underline{k}\wedge \underline{a_1}\wedge\underline{a_2}\wedge\ldots,\\
&\psi_k^*(\underline{a_1}\wedge\underline{a_2}\wedge\ldots):=\sum_{i=1}^\infty(-1)^{i-1}\delta_{a_i,k}\,\underline{a_1}\wedge\ldots\wedge\widehat{\underline{a_i}}\wedge\ldots.
\end{align*}
These operators satisfy the anti-commutation relations
$$
\psi_i\psi_j^*+\psi_j^*\psi_i=\delta_{i,j},\qquad \psi_i\psi_j+\psi_j\psi_i=\psi_i^*\psi_j^*+\psi_j^*\psi_i^*=0.
$$
\emph{Normally ordered} products are defined by
$$
:\psi_i\psi_j^*:=
\begin{cases}
\psi_i\psi_j^*,&\text{if $j>0$},\\
-\psi_j^*\psi_i,&\text{if $j<0$}.
\end{cases}
$$

\medskip

Let $v_\emptyset:=\underline{-\frac{1}{2}}\wedge\underline{-\frac{3}{2}}\wedge\ldots$. For an operator $A\colon\Lambda^{\frac{\infty}{2}}V\to \Lambda^{\frac{\infty}{2}}V$ denote by $\<A\>$ the coefficient of~$v_\emptyset$ in the decomposition of $A(v_\emptyset)$ in the basis
$$
\left\{\underline{a_1}\wedge\underline{a_2}\wedge\ldots\in\Lambda^{\frac{\infty}{2}}V\left|a_1>a_2>\ldots, \text{ $a_i=-i+\frac{1}{2}+c$ for some $c\in\mbC$ and $i$ big enough}\right.\right\}.
$$

\medskip

For any $r\in\mbZ$, define an operator $\cE_r(z)$ on $\Lambda^{\frac{\infty}{2}}V$ by
$$
\cE_r(z)=\sum_{k\in\mbZ+\frac{1}{2}}e^{z(k-\frac{r}{2})}:\psi_{k-r}\psi_k^*:+\frac{\delta_{r,0}}{\varsigma(z)}.
$$
These operators satisfy the commutation relation  
$$
[\cE_a(z),\cE_b(w)]=\varsigma(aw-bz)\cE_{a+b}(z+w).
$$
Finally, we define operators $\alpha_k:=\cE_k(0)$ for $k\ne 0$. We have the commutation relations
\begin{align*}
&[\alpha_a,\cE_b(z)]=\varsigma(az)\cE_{a+b}(z),&& a\ne 0,\\
&[\alpha_a,\alpha_b]=a\delta_{a+b,0},&& a,b\ne 0.
\end{align*}

\medskip

By \cite[Proposition~3.1]{OP06}, we have
$$
F^\bullet_{A,(a_1,\ldots,a_k)}(z_1,\ldots,z_n)=\frac{1}{A\prod_{j=1}^k a_j}\<\alpha_A\prod_{i=1}^n\cE_0(z_i)\prod_{j=1}^k\alpha_{-a_j}\>,\quad k,n\ge 1.
$$
In order to transform the expression $\<\alpha_A\prod_{i=1}^n\cE_0(z_i)\prod_{j=1}^k\alpha_{-a_j}\>$, we move the operators $\alpha_{-a_j}$ through the operators $\cE_0(z_i)$ to the left, using the commutation relation
$$
[\cE_r(z),\alpha_{-k}]=\varsigma(kz)\cE_{r-k}(z).
$$
We obtain
\begin{align*}
\frac{1}{A\prod_{j=1}^k a_j}\<\alpha_A\prod_{i=1}^n\cE_0(z_i)\prod_{j=1}^k\alpha_{-a_j}\>=&\sum_{I_1\sqcup\ldots\sqcup I_n=[k]}\frac{\prod_{l=1}^n\prod_{i\in I_l}\varsigma(a_i z_l)}{A\prod_{j=1}^k a_j}\<\alpha_A\cE_{-A_{I_1}}(z_1)\cdots\cE_{-A_{I_n}}(z_n)\>+\\
&+\frac{\delta_{k,1}}{a_1^2}\<\alpha_{a_1}\alpha_{-a_1}\cE_0(z_1)\cdots\cE_0(z_n)\>.
\end{align*}

\medskip

In order to transform the expression $\<\alpha_A\cE_{-A_{I_1}}(z_1)\cdots\cE_{-A_{I_n}}(z_n)\>$, we move the operators~$\cE_{-A_{I_i}}(z_i)$ through the operator $\alpha_A$ to the left, using the commutation relation
$$
[\cE_a(z),\cE_b(w)]=\varsigma(aw-bz)\cE_{a+b}(z+w).
$$
At the first step, we obtain
\begin{align*}
&\<\alpha_A\cE_{-A_{I_1}}(z_1)\cdots\cE_{-A_{I_n}}(z_n)\>=\\
=&\varsigma(Az_1)\<\cE_{A-A_{I_1}}(z_1)\cE_{-A_{I_2}}(z_2)\cdots\cE_{-A_{I_n}}(z_n)\>+\delta_{I_1,\emptyset}\<\cE_0(z_1)\alpha_A\cE_{-A_{I_2}}(z_2)\cdots\cE_{-A_{I_n}}(z_n)\>.
\end{align*}
In the same way, at each step the number of summands doubles. After $n$ steps, we obtain~$2^n$ summands that are in one-to-one correspondence with subsets $J\subset[n]$: the summand corresponding to a subset $J\subset[n]$ contains the coefficient $\prod_{j\in J}\delta_{I_j,\emptyset}$. Note that the summand corresponding to $J=[n]$ vanishes, because at least one from the subsets $I_i$ is nonempty. The summand corresponding to the subset $J=\emptyset$ is equal to
$$
Q(A_{I_1},\ldots,A_{I_n};z_1,\ldots,z_n),
$$
where
\begin{align*}
&Q(b_1,\ldots,b_n;z_1,\ldots,z_n):=\frac{\scriptstyle{\varsigma(B z_1)\varsigma((B-b_1)z_2+b_2 z_1)\varsigma((B-b_1-b_2)z_3+b_3(z_1+z_2))\cdots\varsigma((B-b_1-\ldots-b_{n-1})z_n+b_n(z_1+\ldots+z_{n-1}))}}{\scriptstyle{\varsigma(z_1+\ldots+z_n)}}.
\end{align*}
In total, we obtain
$$
\<\alpha_A\cE_{-A_{I_1}}(z_1)\cdots\cE_{-A_{I_n}}(z_n)\>=\sum_{J\subsetneq[n]}\left(\prod_{j\in J}\frac{\delta_{I_j,\emptyset}}{\varsigma(z_j)}\right)Q(A_{I_{j_1}},\ldots,A_{I_{j_{|J^c|}}};z_{j_1},\ldots,z_{j_{|J^c|}}),
$$
where $J^c=\{j_1,\ldots,j_{|J^c|}\}$, $j_1<j_2<\ldots<j_{|J^c|}$. As a result,
\begin{align*}
&F^\bullet_{A,(a_1,\ldots,a_k)}(z_1,\ldots,z_n)=\\
=&\sum_{J\subsetneq[n]}\frac{1}{\prod_{j\in J}\varsigma(z_j)}\sum_{I_1\sqcup\ldots\sqcup I_{|J^c|}=[k]}\frac{\prod_{l=1}^{|J^c|}\prod_{i\in I_l}\varsigma(a_i z_{j_l})}{a_1\cdots a_k}\frac{Q(A_{I_1},\ldots,A_{I_{|J^c|}};z_{j_1},\ldots,z_{j_{|J^c|}})}{A}+\\
&+\frac{\delta_{k,1}}{a_1^2}\<\alpha_{a_1}\alpha_{-a_1}\cE_0(z_1)\cdots\cE_0(z_n)\>.
\end{align*}

\medskip

Now we need to take the coefficient of $a_1\cdots a_k$. Note that 
$$
\frac{\prod_{l=1}^{|J^c|}\prod_{i\in I_l}\varsigma(a_i z_{j_l})}{a_1\cdots a_k}=z_{j_1}^{|I_1|}\cdots z_{j_{|J^c|}}^{|I_{|J^c|}|}\prod_{l=1}^{|J^c|}\prod_{i\in I_l}S(a_i z_{j_l}).
$$
Therefore, 
\begin{align*}
&\frac{1}{k!}\Coef_{a_1\cdots a_k}F^\bullet_{A,(a_1,\ldots,a_k)}(z_1,\ldots,z_n)=\\
=&\sum_{J\subsetneq[n]}\frac{1}{\prod_{j\in J}\varsigma(z_j)}\sum_{I_1\sqcup\ldots\sqcup I_{|J^c|}=[k]}z_{j_1}^{|I_1|}\cdots z_{j_{|J^c|}}^{|I_{|J^c|}|}\frac{1}{k!}\Coef_{a_1\cdots a_k}\left(\frac{Q(A_{I_1},\ldots,A_{I_{|J^c|}};z_{j_1},\ldots,z_{j_{|J^c|}})}{A}\right).
\end{align*}
Let us fix $J\subsetneq[n]$ and denote $m=|J^c|$. Then we compute
\begin{align*}
&\sum_{I_1\sqcup\ldots\sqcup I_m=[k]}z_{j_1}^{|I_1|}\cdots z_{j_m}^{|I_m|}\frac{1}{k!}\Coef_{a_1\cdots a_k}\left(\frac{Q(A_{I_1},\ldots,A_{I_m};z_{j_1},\ldots,z_{j_m})}{A}\right)=\\
=&\sum_{I_1\sqcup\ldots\sqcup I_m=[k]}z_{j_1}^{|I_1|}\cdots z_{j_m}^{|I_m|}\frac{|I_1|!\cdots|I_m|!}{k!}\Coef_{b_1^{|I_1|}\cdots b_m^{|I_m|}}\left(\frac{Q(b_1,\ldots,b_m;z_{j_1},\ldots,z_{j_m})}{b_1+\ldots+b_m}\right)=\\
=&\sum_{k_1+\ldots+k_m=k}z_{j_1}^{k_1}\cdots z_{j_m}^{k_m}\Coef_{b_1^{k_1}\cdots b_m^{k_m}}\left(\frac{Q(b_1,\ldots,b_m;z_{j_1},\ldots,z_{j_m})}{b_1+\ldots+b_m}\right)=\\
=&\Coef_{t^k}\left(\frac{Q(t z_{j_1},\ldots,t z_{j_m};z_{j_1},\ldots,z_{j_m})}{t Z_{J^c}}\right)=\\
=&\frac{1}{Z_{J^c}}\Coef_{t^{k+1}}Q(t z_{j_1},\ldots,t z_{j_m};z_{j_1},\ldots,z_{j_m}).
\end{align*}
It remains to note that 
$$
Q(t z_{j_1},\ldots,t z_{j_m};z_{j_1},\ldots,z_{j_m})=\frac{\prod_{i\in J^c}\varsigma(z_j Z_{J^c}t)}{\varsigma(Z_{J^c})},
$$
which completes the proof of equality~\eqref{eq:formula2}.

\medskip

\section*{Declarations}

\subsection*{Conflict of interest statement} On behalf of all authors, the corresponding author states that there is no conflict of interest.

\medskip

\subsection*{Data availability} This manuscript has no associated data.

\end{document}